\documentclass[12pt,leqno]{amsart}
\usepackage{amsmath,amsthm,amsfonts,latexsym,amssymb,leqno}
\usepackage[utf8]{inputenc}
\usepackage{graphicx}
\usepackage{float}
\usepackage{tikz}
\usetikzlibrary{decorations.pathreplacing}
\theoremstyle{plain} 
\newtheorem{thm}{Theorem}[section]
\newtheorem{cor}[thm]{Corollary}
\newtheorem{lm}[thm]{Lemma}

\newtheorem{clm}[thm]{Claim}
\newtheorem{subclm}[thm]{Subclaim}
\theoremstyle{definition}
\newtheorem{df}[thm]{Definition}

\numberwithin{equation}{section}
\newcommand{\m}[1]{{\mathbf{\uppercase{#1}}}}
\newcommand{\Con}{\mathrm{Con}}
\newcommand{\Sub}{\mathsf{S}}
\newcommand{\Prod}{\mathsf{P}}
\newcommand{\solv}{\stackrel{s}{\sim}}
\newcommand{\typ}{{\mathrm{typ}}}

\newcommand{\qedsymbdiamond}{\renewcommand{\qedsymbol}{$\diamond$}}

\begin{document}

\title[Neutrabelian algebras]{Neutrabelian algebras}

\author{Keith A. Kearnes}
\address[Keith A. Kearnes]{Department of Mathematics\\
University of Colorado\\
Boulder, CO 80309-0395\\
USA}
\email{Keith.Kearnes@Colorado.EDU}

\author{Connor Meredith}
\address[Connor Meredith]{Department of Mathematics\\
University of Colorado\\
Boulder, CO 80309-0395\\
USA}
\email{Connor.Meredith@Colorado.EDU}

\author{\'Agnes Szendrei}
\address[\'Agnes Szendrei]{Department of Mathematics\\
University of Colorado\\
Boulder, CO 80309-0395\\
USA}
\email{Szendrei@Colorado.EDU}

\thanks{This material is based upon work supported by
  the National Science Foundation grant no.\ DMS 1500254,
  the Hungarian National Foundation for Scientific Research (OTKA)
  grant no.\ K115518, and
  the National Research, Development and Innovation Fund of Hungary (NKFI)
  grant no.\ K128042.}


\begin{abstract}
  We introduce ``neutrabelian algebras'',
  and prove that finite, hereditarily neutrabelian algebras
  with a cube term are dualizable.
\end{abstract}

\maketitle

\section{Introduction}
This paper investigates the relationship between
two commutator properties
for finite algebras in congruence modular varieties.
The investigation is motivated by the problem
of determining which finite algebras in residually
small congruence modular varieties are dualizable in the sense
of \cite{clark-davey}, i.e., which finite algebras
in residually
small congruence modular varieties may serve
as the character algebra for a natural duality.

The first commutator property, new to this paper,
is the property of being {\bf neutrabelian}.
This word is a portmanteau of \emph{neutral} and
\emph{abelian}. Recall that if
$\m a$ is any algebra in a congruence modular variety,
then
\begin{equation}\label{monotonicity}
0\leq [\alpha,\beta]\leq \alpha\wedge\beta
\end{equation}
for any two congruences $\alpha$ and $\beta$ on $\m a$.
The (universally quantified) commutator identity (C4),
\[
[\alpha,\beta]=0,
\]
represents one extreme type of commutator behavior,
where the left inequality in (\ref{monotonicity}) is equality
for all $\alpha, \beta\in\Con(\m a)$.
An algebra satisfying (C4) is called {\bf abelian}.
The other extreme behavior,
where the right inequality in (\ref{monotonicity}) is equality
for all $\alpha, \beta\in\Con(\m a)$,
is represented by the commutator identity (C3),
\[
[\alpha,\beta]=\alpha\wedge\beta.
\]
An algebra satisfying (C3) is called {\bf neutral}.

The neutrabelian property is a combination of (C3) and (C4).

\begin{df} \label{neutrabelianSI}
  Let $\m a$ be a subdirectly irreducible algebra (= SI)
  in a congruence modular variety. Let $\mu$
  be the monolith of $\m a$ and let $\nu=(0:\mu)$ be the centralizer
  of the monolith. $\m a$ is a
{\bf neutrabelian SI} if $\nu$ is comparable
with all congruences of $\m a$, and the commutator operation
on $\Con(\m a)$ satisfies
\begin{enumerate}
\item[``(C4)''] $[\alpha,\beta] = 0$ if $\alpha$ and $\beta$ are both
  contained in $\nu$, and
\item[``(C3)''] $[\alpha,\beta] = \alpha\wedge \beta$ otherwise.
\end{enumerate}
\end{df}

In words, this says that a neutrabelian SI is 
a subdirectly irreducible algebra where $\nu$
is comparable to all other congruences and the commutator
satisfies {\rm (C4)} below $\rho$ and {\rm (C3)} elsewhere.
An SI algebra is abelian if and only if it is
a neutrabelian SI
with $\nu=1$, while an SI algebra is
neutral if and only if it is
a neutrabelian SI with $\nu=0$.

The shape of $\Con(\m a)$, when $\m a$ is a neutrabelian SI,
is depicted in Figure~\ref{fig1}. The lattice on the left
indicates the situation where the monolith of $\m a$ is nonabelian,
while the lattice on the right indicates the situation where
the monolith is abelian.

\begin{figure}[!htbp]
\setlength{\unitlength}{1truemm}
\begin{picture}(80,70)
\put(-10,5)
{
\begin{tikzpicture}[scale=.8]
\draw [fill] (1,0) circle [radius=0.1];
\draw [fill] (1,1) circle [radius=0.1];
\draw [fill] (1,7) circle [radius=0.1];
\draw[line width=1.2pt] (1,0) -- (1,1);
\draw[line width=1.2pt] (1,4) ellipse (1 and 3);

\node at (2,0){$0=\nu$};
\node at (1.8,1){$\mu$};
\node at (1.8,7){$1$};

\draw[line width=1.2pt] [decorate,decoration={brace,amplitude=10pt,mirror,raise=4pt},yshift=0pt]
(3.2,0) -- (3.2,7) node [black,midway,xshift=1.5cm] {\footnotesize
neutral};

\draw [fill] (9,0) circle [radius=0.1];
\draw [fill] (9,1) circle [radius=0.1];
\draw [fill] (9,4) circle [radius=0.1];
\draw [fill] (9,7) circle [radius=0.1];
\draw[line width=1.2pt] (9,0) -- (9,1);
\draw[line width=1.2pt] (9,2.5) ellipse (1 and 1.5);
\draw[line width=1.2pt] (9,5.5) ellipse (1 and 1.5);

\node at (9.8,0){$0$};
\node at (9.8,1){$\mu$};
\node at (9.8,4){$\nu$};
\node at (9.8,7){$1$};

\draw[line width=1.2pt] [decorate,decoration={brace,amplitude=10pt,mirror,raise=4pt},yshift=0pt]
(10.5,0) -- (10.5,3.9) node [black,midway,xshift=1.5cm] {\footnotesize
abelian};
\draw[line width=1.2pt] [decorate,decoration={brace,amplitude=10pt,mirror,raise=4pt},yshift=0pt]
(10.5,4.1) -- (10.5,7) node [black,midway,xshift=1.5cm] {\footnotesize
neutral};

\end{tikzpicture}
}
\end{picture}
\begin{caption}
  {\sc $\Con(\m a)$, $\mu$ nonabelian versus $\mu$ abelian} \label{fig1}
\end{caption}
\end{figure}

In the left figure, the shape of $\Con(\m a)$
is not restricted by the definition
of ``neutrabelian SI'', except
for the fact that neutral algebras
in congruence modular varieties have distributive congruence
lattices. However, in the right figure, where $\nu>0$,
the shape of $\Con(\m a)$ is restricted.
This lattice is articulated at $\nu$.
We have indicated that $\m a$ is neutral above $\nu$ and abelian
below $\nu$, but in fact more is true:
if $\beta>\nu$ and $\alpha$ is arbitrary, we have
$[\alpha,\beta]=\alpha\wedge\beta$.

  The articulation at $\nu$
  may seem to be a strange condition to put on $\Con(\m a)$,
but it is a shape that occurs in nature.
The fact that any SI unital ring in a
residually small variety is a neutrabelian SI follows from
\cite{mckenzie}.
The fact that any
SI algebra in a finitely generated, finitely decidable variety
is a neutrabelian SI follows from \cite{idziak}.
Any dualizable,
endorigid, SI algebra in a congruence permutable variety
is a neutrabelian SI, as can be shown using ideas in \cite{idziak0}.

\begin{df} \label{neutrabelian}
  Let $\m a$ be an algebra
  in a congruence modular variety. $\m a$ is 
  {\bf neutrabelian} if every SI quotient of $\m a$
  is a neutrabelian SI according to Definition~\ref{neutrabelianSI}.
\end{df}

One must now ask whether this definition is consistent
with Definition~\ref{neutrabelianSI}. Namely, is it true that if
$\m a$ is a neutrabelian SI according to
Definition~\ref{neutrabelianSI}, then
must $\m a$ be 
neutrabelian according to
Definition~\ref{neutrabelian}?
To establish an affirmative answer to this question
one must verify that any SI
quotient $\m a/\delta$ of a neutrabelian SI $\m a$
is also a neutrabelian SI. Argue as follows:
suppose that $\m a$ is a neutrabelian SI and $\nu$
is the centralizer of the monolith of $\m a$.
If $\delta\geq \nu$, then $\m a/\delta$ is neutral
and SI, hence a neutrabelian SI.
If $\delta < \nu$, then it follows from
conditions ``(C4)'' and ``(C3)'' of
Definition~\ref{neutrabelianSI} and properties
of the modular commutator that $\nu/\delta$
is the centralizer of the monolith of $\m a/\delta$
and conditions ``(C4)'' and ``(C3)'' of
Definition~\ref{neutrabelianSI} hold for
$\m a/\delta$.

Next we describe the second commutator
property that is the focus of this paper, the
split centralizer condition.
Let $\m a$ be a finite algebra and let $\mathcal Q$
be a quasivariety containing $\m a$.
A {\bf $\mathcal Q$-congruence}
on $\m a$ is a congruence $\kappa$ of $\m a$
such that $\m a/\kappa\in \mathcal Q$.
Let $\delta$ be a completely meet irreducible congruence on $\m a$
with upper cover $\theta$, and let $\nu = (\delta:\theta)$
be the centralizer of $\theta$ modulo $\delta$.
A triple $(\delta,\theta,\nu)$ defined in this way
is called {\bf relevant} if $\theta\leq \nu$, that is,
if $\theta/\delta$ is abelian.
The triple $(\delta,\theta,\nu)$
is {\bf split} by a triple of congruences $(\alpha,\beta,\kappa)$ $\m a$ if
\begin{enumerate}
\item[(i)] $\kappa$ is a $\mathcal Q$-congruence on $\m a$,
\item[(ii)] $\beta\leq \delta$,
\item[(iii)] $\alpha\wedge\beta = \kappa$,
\item[(iv)] $\alpha\vee\beta = \nu$, and
\item[(v)] 
$\alpha/\kappa$ is abelian.
\end{enumerate}
The relationships between 
these congruences are
depicted in Figure~\ref{fig2}.
\begin{figure}[!htbp]
\setlength{\unitlength}{1truemm}
\begin{picture}(90,62)
\put(10,0){%
\begin{tikzpicture}[scale=.65]
\draw[line width=1.2pt] (-.07,-2.93) -- (-3,0) -- (-1.4,1.6);
\draw[line width=1.2pt] (.07,-2.93) -- (3,0) -- (0,3);
\draw[line width=1.2pt] plot [smooth,tension=1]
               coordinates{(-1.4,1.6) (-.9,2.5) (0,3)};
\draw[line width=1.2pt] plot [smooth,tension=1]
               coordinates{(-1.4,1.6) (-.5,2.1) (0,3)};
\draw[line width=1.2pt] plot [smooth,tension=1]
               coordinates{(-1.4,1.6) (-1.18,3.08) (0,4)};
\draw[line width=1.2pt] plot [smooth,tension=1]
               coordinates{(-1.4,1.6) (.2,2.5) (0,4)};
\draw (0,-3) circle [radius=0.1];    
\draw [fill] (-3,0) circle [radius=0.1];    
\draw [fill] (-1.8,1.2) circle [radius=0.1]; 
\draw [fill] (-1.4,1.6) circle [radius=0.1];   
\draw [fill] (0,3) circle [radius=0.1];   
\draw [fill] (3,0) circle [radius=0.1];
\node at (.4,-3){\small $\kappa$};
\node at (-3.4,0){\small $\beta$};
\node at (-2.2,1.25){\small $\delta$};
\node at (-1.8,1.65){\small $\theta$};
\node at (0,3.3){\small $\nu$};
\node at (3.4,0){\small $\alpha$};
%
\draw[line width=1.2pt] plot [smooth,tension=1.5]
               coordinates{(0,-4) (-4,0) (0,4)};
\draw[line width=1.2pt] plot [smooth,tension=1.5]
               coordinates{(0,-4) (4,0) (0,4)};
\draw [fill] (0,4) circle [radius=0.1];
\draw [fill] (0,-4) circle [radius=0.1];
\node at (0,4.5){$1$};
\node at (0,-4.5){$0$};
%
\node at (-4,3.5){$\Con(\m a)$};
\node at (7,3){$\nu=(\delta:\theta)$};
\node at (6.5,.5){$\beta\leq\delta$};
\node at (7,-.5){$\alpha\wedge\beta=\kappa$};
\node at (7,-1.5){$\alpha\vee\beta=\nu$};
\node at (7,-2.5){$[\alpha,\alpha]\leq\kappa$};
\end{tikzpicture}
}
\end{picture}
\caption{}\label{fig2}
\end{figure}
We say that a relevant triple $(\delta,\theta,\nu)$
is {\bf split at $0$} if it is split by a triple
of the form $(\alpha,\beta,\kappa)$ satisfying the above conditions
with $\kappa=0$.

In \cite{kearnes-szendrei}, the
{\bf split centralizer condition} 
for a finite algebra
$\m a$ was defined
to be the condition that, for $\mathcal Q := \Sub\Prod(\m a)$
and for any subalgebra $\m b\leq\m a$, each
relevant triple $(\delta,\theta,\nu)$ of $\m b$
is split (relative to $\mathcal Q$) by
some triple $(\alpha,\beta,\kappa)$.
Here we shall consider a modified version
of this condition, which allows us to
ignore the role of the quasivariety
$\mathcal Q$. Namely, we shall only consider
the situation where relevant triples $(\delta,\theta,\nu)$
are split at $0$. Since $0$ is a $\mathcal Q$-congruence
for any $\mathcal Q$ containing
$\m a$, we will be able to ignore the role of $\mathcal Q$.
To make our assumptions explicit,
we record this as:

\begin{df} \label{funnysplit}
  Say that $\m a$ has {\bf centralizers split at $0$}
  if every relevant triple $(\delta,\theta,\nu)$ of $\m a$
  is split by a triple of the form
  $(\alpha,\beta,0)$.
\end{df}

If $\m a_A$ is the expansion of $\m a$
by constants, then $\m a_A$
satisfies the split centralizer condition
as defined in \cite{kearnes-szendrei} if and only if
$\m a_A$ has centralizers split at $0$ as defined
in Definition~\ref{funnysplit}.
But without expanding $\m a$ by constants
we do not have equivalence.
We have only that if $\m a$ and all subalgebras of $\m a$ have
centralizers split at $0$, then
$\m a$ satisfies the split centralizer condition
as defined in \cite{kearnes-szendrei}, but in the
strong sense that the $\mathcal Q$-congruence
$\kappa$ that arises in a splitting
triple $(\alpha,\beta,\kappa)$
must be the zero congruence.

Our goal in this paper is to
prove the following theorem.

\begin{thm}\label{split_cent_thm}
The following are equivalent for a finite algebra
$\m a$ 
in a congruence modular
variety.
\begin{enumerate}
\item[(1)]
$\m a$ is neutrabelian.
\item[(2)]
  $\m a$ has centralizers split at $0$.
\end{enumerate}  
\end{thm}

Property (2) of this theorem is a statement that
certain congruences, namely centralizers of abelian prime
quotients, have a special kind of ``direct factorization''.
Property (1) of this theorem is a statement
about the commutator structure of the subdirect factors of $\m a$.
When judging whether a theorem of this sort
is interesting one might reflect on whether
it is easier to verify if special direct factorizations
exist or it is easier
to verify if subdirect factors have given
commutator properties.
Our motivation for proving Theorem~\ref{split_cent_thm}
was the feeling that it is hard
to establish direct decompositions of congruences,
and so Property (1) is more illuminating than Property (2).

The relevance of Theorem~\ref{split_cent_thm}
for natural duality theory is the following
corollary.

\begin{cor}\label{split_cent_cor}
  If $\m a$ is a finite algebra with
  a cube term, and every subalgebra of $\m a$
  is neutrabelian, then $\m a$ is dualizable.
\end{cor}

\begin{proof}
If every subalgebra of $\m a$
is neutrabelian, then $\m a$ satisfies the split centralizer condition
as defined in \cite{kearnes-szendrei}. In
Theorem~1.1 of \cite{kearnes-szendrei} it is proved
that any finite algebra with a cube term which
satisfies the split centralizer condition is dualizable.
\end{proof}

This yields the first proof of the following fact.

\begin{cor}
  If $\m a$ is a finite algebra in a finitely decidable
congruence modular variety, then $\m a$ is dualizable.
\end{cor}

\begin{proof}
  It follows from the Main Theorem of
  \cite{idziak} that if $\mathcal V$ is a finitely decidable
  congruence modular variety, then $\mathcal V$
  has a Maltsev term and every finite algebra in $\mathcal V$
  is neutrabelian.
  Since Maltsev terms are special types of cube terms,
  this corollary follows from Corollary~\ref{split_cent_cor} above.
\end{proof}

\section{The equivalence of two conditions}\label{equivalence}

Recall that the commutator identity (C1) holds for $\m a$ 
exactly when, for every $\alpha, \beta\in\Con(\m a)$, it is the case that

\begin{align*}
[\alpha\wedge \beta,\beta] &= \alpha\wedge [\beta,\beta].\tag{C1} \\  
\end{align*}

\noindent
Consider the following related condition
for an SI, $\m s$, with monolith $\mu$:
say that $\m s$ satisfies (C1)${}^*$ if the
centralizer of its monolith, $(0:\mu)$,
is abelian. It is proved in the opening
paragraphs of Section 2 of
\cite{freese-mckenzie0} that if $\m a$ belongs
to a congruence modular variety, then 
$\m a$ satisfies (C1) if and only if every SI quotient
of $\m a$ satisfies (C1)${}^*$. 

Recall also that the {\bf solvable radical} of a congruence
$\tau\in\Con(\m a)$ is the largest congruence
$\rho(\tau)$ for which the interval $I[\tau,\rho(\tau)]$
is solvable, if a largest such congruence exists.
The {\bf solvable radical} of the algebra
$\m a$ is $\rho:=\rho(0)$, the largest
solvable congruence, if such exists.
It follows from the properties of the modular commutator
that if $\tau \leq \alpha, \beta$ and both $\alpha$
and $\beta$ are solvable over $\tau$,
then $\alpha\vee \beta$ is solvable over $\tau$.
Thus, if $\Con(\m a)$ is finite, the join
all congruences solvable over $\tau$ equals $\rho(\tau)$.

We can use (C1) and the concept of the solvable radical to
express the neutrabelian property for SIs
in a slightly simpler way.

\begin{lm} \label{neut_char}
  Let $\m a$ be a finite SI algebra in a congruence modular variety.
  $\m a$ is neutrabelian if and only if
  \begin{enumerate}
\item[(1)]  $\m a$ satisfies (C1).
\item[(2)]  The solvable radical $\rho$ is comparable with all
  congruences of $\m a$.
\item[(3)]  $\rho$ is abelian.
\item[(4)]  $\m a/\rho$ is neutral.
\end{enumerate}
\end{lm}  

\begin{proof}
For both directions of the argument
assume that $\m a$ is a finite SI with monolith $\mu$ and centralizer
$\nu=(0:\mu)$.

Now assume that $\m a$ is neutrabelian.
For any two congruences $\alpha,\beta\in\Con(\m a)$
it is the case that 
$[\alpha\wedge\beta,\beta]\leq \alpha\wedge[\beta,\beta]$.
If the inequality is strict, then $0\neq [\beta,\beta]$,
so $\nu < \beta$ by ``(C4)''. But then, by ``(C3)'',
we get
$[\alpha\wedge\beta,\beta]=(\alpha\wedge\beta)\wedge\beta$
and  $\alpha\wedge[\beta,\beta] = \alpha\wedge\beta$,
and these are equal, so the inequality cannot be strict
after all. This means that (C1) holds.

In a neutrabelian SI, $\nu$ is abelian,
and $\m a/\nu$ is neutral. This is enough to
imply that $\nu = \rho$ = the solvable radical
of $\m a$. Hence (2), (3), (4) hold
in a neutrabelian SI by Definition~\ref{neutrabelianSI}.

Now we prove the reverse direction.
Assume that $\m a$ is SI
and satisfies (1)--(4). Since $\m a$ satisfies (C1),
it satisfies (C1)${}^*$, so $\nu$ is abelian.
This yields $\nu\leq \rho$. On the other hand,
$\rho$ is abelian (as we assume in (3)),
so $[\rho,\mu]=0$, either
because $\rho=0$ or because $\mu\leq \rho$.
Hence $\rho\leq\nu$. This yields $\rho=\nu$.
Hence (2) and (3) imply that $\nu$
is abelian and is comparable with all congruences.
What remains to show is that $[\alpha,\beta]=\alpha\wedge \beta$
if $\nu < \beta$.

If $\nu<\alpha\wedge\beta$, then we get 
$[\alpha,\beta]=\alpha\wedge \beta$ from the neutrality of
$\m a/\nu=\m a/\rho$ (Item~(4)). So the only situation left
to check is when $\alpha\wedge\beta\leq \nu=\rho < \beta$.
In this situation we have
\begin{align*}
[\alpha\wedge\beta,\beta] \stackrel{\rm(C1)}{=} \alpha\wedge[\beta,\beta] =
\alpha\wedge\beta,   \tag{$\dagger$}  
\end{align*}
since $[\beta,\beta]=\beta$ when $\rho<\beta$. But then
\begin{align*}
[\alpha\wedge\beta,\beta] \leq [\alpha,\beta]\leq
\alpha\wedge\beta\stackrel{(\dagger)}{=}
[\alpha\wedge\beta,\beta],\tag{$\ddagger$}  
\end{align*}
so we have equality in the middle of $(\ddagger)$,
$[\alpha,\beta]=\alpha\wedge\beta$.
\end{proof}

The characterization given by Lemma~\ref{neut_char}
is very close to the definition of
``neutrabelian SI'', so let us take a moment to emphasize the difference.
In the definition, we specify that $[\alpha,\beta]=\alpha\wedge \beta$
if at least one of $\alpha, \beta$ is strictly above $(0:\mu)$.
Hence to verify the definition holds for some algebra we may have to examine
values of $[\alpha,\beta]$ in some cases where
$\alpha < (0:\mu) < \beta$. The lemma allows us to avoid
examining these values. Given (C1) and the comparability
of $(0:\mu)$ with all congruences, it suffices
to check the values of $[\alpha,\beta]$ when both
$\alpha,\beta$ lie below $(0:\mu)$ or when both
lie above $(0:\mu)$.

\begin{lm} \label{properties}
  Let $\m a$ be a finite algebra in a congruence modular
  variety. Under either one of the hypotheses
\begin{enumerate}
\item[(i)]
    $\m a$ is neutrabelian, or
\item[(ii)]
  $\m a$ has centralizers split at $0$, 
\end{enumerate}
the following conclusions hold:
\begin{enumerate}
\item[(1)] $\m a$ satisfies the commutator identity {\rm{(C1)}}.
\item[(2)] $\m a$ satisfies the
  $\langle {\bf 3}, {\bf 2}\rangle$ and 
  $\langle {\bf 4}, {\bf 2}\rangle$-transfer principles
  of tame congruence theory.  
\item[(3)] The solvable radical $\rho$ is the largest
  abelian congruence of $\m a$, and
  $\m a/\rho$ is neutral.
\item[(4)] If $\tau$ is a congruence of $\m a$, then the
  solvable radical $\rho(\tau)$ of $\tau$ equals $\rho\vee\tau$,
  the interval $I[\tau,\rho(\tau)]$ is abelian,
  and the interval $I[\rho(\tau),1]$ is neutral.
\end{enumerate}
\end{lm}

\begin{proof}
  In the first part of the proof we assume
  (i), which is the assumption
  that $\m a$ is a neutrabelian algebra,
  and we argue that each of $(1)$--$(4)$ hold.

  An algebra is neutrabelian
  if and only if each SI quotient is, and an algebra satisfies
  (C1) if and only if each SI quotient does. 
  Hence, to prove Item (1), it suffices to prove that any
  neutrabelian SI
  satisfies (C1), which we did in Lemma~\ref{neut_char}.
  
Now consider Item (2).
The $\langle {\bf i}, {\bf j}\rangle$-transfer
principle for a finite algebra $\m a$ ``is''
(or is equivalent to) the statement that $\Con(\m a)$
has no $3$-element interval $\alpha\prec\beta\prec\gamma$
with $\typ(\alpha,\beta)={\bf i}$ and
with $\typ(\beta,\gamma)={\bf j}$. We need the following claim.

\begin{clm} \label{transfer}
Assume that $\m a$ is a finite algebra in a congruence modular
  variety. If $\Con(\m a)$ has a
$3$-element interval $\alpha\prec\beta\prec\gamma$
with $\typ(\alpha,\beta)={\bf i}$ and
with $\typ(\beta,\gamma)={\bf j}$, then
some SI quotient of $\m a$ has congruences
$0\prec\mu\prec\sigma$
with $\typ(0,\mu)={\bf i}$ and
with $\typ(\mu,\sigma)={\bf j}$.
\end{clm}

\begin{proof}[Proof of Claim]
This claim is derivable from
Theorem 3.13 of \cite{agliano-kearnes}, where
a similar claim is asserted in the more general
context of congruence semimodular
varieties. Nevertheless, we include a proof here.
Assume that, for some ${\bf i}$ and ${\bf j}$,
$\Con(\m a)$ has some $3$-element interval
$I[\alpha,\gamma]$ with 
$\alpha\stackrel{\bf i}{\prec}\beta\stackrel{\bf j}{\prec}\gamma$.
Choose a congruence $\delta\geq \alpha$
that it is maximal for the property $\delta\not\geq \beta$.
Necessarily $\delta$ is completely meet irreducible
with upper cover $\delta\vee\beta$. Since $\delta\geq\alpha$ and
$\delta\not\geq\beta$,
we have $\delta\wedge\gamma\geq\alpha$
and $\delta\wedge \gamma\not\geq \beta$. Since the
interval $I[\alpha,\gamma]$ contains only $\alpha, \beta, \gamma$,
it follows that
$\delta\wedge\gamma = \alpha=\delta\wedge\beta$.
By modularity it follows that
$\delta\vee\gamma\neq \delta\vee\beta$.
By tame congruence theory
we even get 
$
\delta
\stackrel{\bf i}{\prec}
\delta\vee\beta
\stackrel{\bf j}{\prec}
\delta\vee\gamma
$.
In the SI quotient $\m a/\delta$, with monolith
$\mu = (\delta\vee\beta)/\delta$ and upper cover
$\sigma = (\delta\vee\gamma)/\delta$ this yields
$0\stackrel{\bf i}{\prec} \mu\stackrel{\bf j}{\prec} \sigma$.
\qedsymbdiamond
\end{proof}

\bigskip

Now we complete the proof of Item (2).
The types that can appear in $\Con(\m a)$
when $\m a$ belongs to congruence modular variety
are ${\bf 2}, {\bf 3}, {\bf 4}$,
and $\alpha\stackrel{\bf 2}{\prec} \beta$
if and only if $\beta$ is abelian over $\alpha$.
In a neutrabelian SI, if the monolith is nonabelian
(type ${\bf 3}$ or ${\bf 4}$), then the SI is neutral,
hence all congruence coverings are nonabelian.
This shows that
$0\stackrel{\bf 3}{\prec} \mu\stackrel{\bf 2}{\prec} \sigma$ and
$0\stackrel{\bf 4}{\prec} \mu\stackrel{\bf 2}{\prec} \sigma$
are forbidden in any finite neutrabelian SI.
By Claim~\ref{transfer}, $3$-element intervals
$\alpha\stackrel{\bf 3}{\prec}\beta\stackrel{\bf 2}{\prec}\gamma$ and
$\alpha\stackrel{\bf 4}{\prec}\beta\stackrel{\bf 2}{\prec}\gamma$
are forbidden in any finite neutrabelian algebra.
This is the statement that $\m a$ satisfies the
$\langle {\bf 3}, {\bf 2}\rangle$ and
$\langle {\bf 4}, {\bf 2}\rangle$-transfer principles.

Next we consider Item (3).
The solvable radical is, by definition,
the largest solvable congruence on $\m a$.
According to tame congruence theory,
the solvable radical of $\m a$
equals the largest congruence $\rho\in \Con(\m a)$ such that
the interval $I[0,\rho]$ contains
only abelian types. Since $\m a$ lies in a congruence
modular variety, where the only abelian
type is ${\bf 2}$, this means
that $\rho$ is the largest congruence on
$\m a$ such that $\typ\{0,\rho\}=\{{\bf 2}\}$.
It follows from the 
$\langle {\bf 3}, {\bf 2}\rangle$ and
$\langle {\bf 4}, {\bf 2}\rangle$-transfer principles
(and the fact that $\typ\{\m a\}\subseteq \{{\bf 2},{\bf 3},{\bf 4}\}$)
that $\typ\{\rho,1\}\subseteq \{{\bf 3}, {\bf 4}\}$.
This implies that $\Con(\m a/\rho)$ has no abelian types,
which is equivalent to the statement that 
$\m a/\rho$ is neutral.

Now we argue that $\rho$ is abelian.
Let $\m a\leq \prod \m a_i$ be a representation 
of $\m a$ as a subdirect product of neutrabelian SIs.
If the solvable congruence $\rho$ were not abelian, then its projection
onto some factor would have to be solvable but not abelian.
Neutrabelian SIs have no solvable congruences
that are not abelian, so we are forced to conclude that $\rho$
is abelian. This completes the proof of Item (3).

Now we prove Item (4).
If $\m a$ is neutrabelian, then any quotient 
$\m a/\tau$ is neutrabelian, hence the radical
$\rho^{\m a/\tau} = \rho(\tau)/\tau$ is abelian
and the quotient $\m a/\rho^{\m a/\tau}$ is neutral
by Item (3). This implies that the interval
$I[\tau,\rho(\tau)]$ in $\Con(\m a)$ is abelian
and the interval $I[\rho(\tau),1]$ is neutral in $\Con(\m a)$.
It remains to prove that $\rho(\tau) = \rho\vee\tau$.

Write $\stackrel{s}{\sim}$ for the solvability relation
of tame congruence theory. Since $0\stackrel{s}{\sim} \rho$ 
in $\Con(\m a)$ and $\stackrel{s}{\sim}$ is a congruence
on $\Con(\m a)$ we get $\tau = 0\vee\tau\stackrel{s}{\sim}\rho\vee\tau$.
This forces $\rho\vee\tau\leq \rho(\tau)$, since
$\rho(\tau)$ is the largest congruence related to
$\tau$ by $\stackrel{s}{\sim}$.
In particular, $I[\rho\vee\tau,\rho(\tau)]$
is a solvable interval.
However, $I[\rho\vee\tau,\rho(\tau)]$
is a subinterval of 
$I[\rho,1]$, which by Item (3) is neutral.
Since it is both solvable and neutral it is trivial.
This forces $\rho(\tau) = \rho\vee\tau$.

\bigskip

We have proved the lemma under hypothesis (i).
Now we prove it under hypothesis (ii),
which is the assumption that $\m a$ has centralizers
split at $0$.

For Item (1), it again
suffices to prove that any SI quotient of $\m a$
satisfies (C1)${}^*$. Let $\m a/\delta$ be an SI
quotient of $\m a$, let $\theta, \nu\geq \delta$ in $\Con(\m a)$
be such that $\theta/\delta$ is the monolith of $\m a/\delta$
and $\nu/\delta=(\delta/\delta:\theta/\delta)$
is the centralizer of $\theta/\delta$.
(This forces $\nu=(\delta:\theta)$.)

There is nothing to prove unless the monolith $\theta/\delta$ is abelian,
so we may assume that
$\delta\stackrel{{\bf 2}}{\prec}\theta\leq \nu$.
In this situation $(\delta,\theta,\nu)$ is a relevant triple of
$\m a$, hence must be split at $0$ by $(\alpha,\beta,0)$
where $\beta\leq\delta$, $\alpha\wedge\beta=0$,
$\alpha\vee\beta=\nu$, and $\alpha$ is abelian. Now
\[
  [\nu,\nu]=[\alpha\vee\beta,\alpha\vee\beta]\leq [\alpha,\alpha]\vee\beta
  \leq 0\vee\delta=\delta,
  \]
  which proves that $\nu$ is abelian over $\delta$.
From this it follows that the centralizer
$(\delta/\delta:\theta/\delta)=\nu/\delta$ of the monolith
$\theta/\delta$ in the quotient algebra $\m a/\delta$
is abelian over $\delta/\delta$, i.e. that (C1)${}^*$ holds.

We temporarily skip Item (2) and prove Item (3).
We will use the fact, proved
in Lemma~4.16 of \cite{mckenzie2}
(also appearing as Exercise 8.5 of \cite{freese-mckenzie}),
that if $\m a$ is an algebra in a congruence
modular variety, and $\m a$
satisfies (C1), then $\m a$ has a largest
abelian congruence, which we denote $\widehat{\alpha}$.

\begin{clm} \label{neutral}
$\m a/\widehat{\alpha}$ is neutral.
\end{clm} 

\begin{proof}[Proof of Claim]
If $\m a/\widehat{\alpha}$ had any abelian prime quotient,
then $\m a$ would have one, $\delta\prec\theta$,
where $\delta$ is completely meet irreducible and
$\widehat{\alpha}\leq \delta\prec\theta$.
If $\nu=(\delta:\theta)$, then necessarily
\[
\widehat{\alpha}\leq \delta\prec\theta\leq\nu\]
so $(\delta,\theta,\nu)$ would be a relevant triple.
If this triple is split by $(\alpha,\beta,0)$,
then since $\alpha$ is abelian we must have $\alpha\leq\widehat{\alpha}$.
Since $(\alpha,\beta,0)$ splits $(\delta,\theta,\nu)$
we get
\[
\nu=\alpha\vee\beta\leq \widehat{\alpha}\vee\delta = \delta < \nu,
\]
which is impossible. This proves the claim.
\qedsymbdiamond
\end{proof}

\bigskip

Claim~\ref{neutral} implies that $\rho=\widehat{\alpha}$
(since $\widehat{\alpha}$ is abelian and $\m a/\widehat{\alpha}$
is neutral). Thus Item~(3) has been proved:
$\rho$ is the largest abelian congruence of $\m a$,
and $\m a/\rho$ is neutral.

We easily infer the truth of Item (4), as well.
If $\tau$ is a congruence of $\m a$, then
since $\rho$ is abelian over $0$ we get that 
$\rho\vee\tau$ is abelian over $0\vee\tau = \tau$.
Since $I[\rho\vee\tau,1]$ is a subinterval of
the neutral interval $I[\rho,1]$,
we get that $I[\rho\vee\tau,1]$ is neutral.
Thus the congruence $1$ is neutral over $\rho\vee\tau$,
while $\rho\vee\tau$ is abelian over $\tau$;
this is enough to conclude that $\rho(\tau)=\rho\vee\tau$.

We conclude by proving Item (2).
If $\m a$ fails the $\langle {\bf i}, {\bf j}\rangle$-transfer
principle for ${\bf i}\in\{{\bf 3}, {\bf 4}\}$
and ${\bf j}={\bf 2}$, then by Claim~\ref{transfer}
there is a completely meet irreducible congruence $\delta$
such that
$\delta\stackrel{{\bf 3},{\bf 4}}{\prec}\theta
\stackrel{\bf 2}{\prec}\sigma$.
Now we use the statement in Item~(4) to get a contradiction.
Since $\delta$ is
completely meet irreducible with nonabelian
upper cover we have $\delta = \rho(\delta) = \rho\vee\delta$.
This implies that $\rho\leq \delta$. 
Since $\theta$ has an abelian cover $\sigma$ we have
$\theta < \rho(\theta) = \rho\vee\theta$.
This implies that $\rho\not\leq\theta$.
But the conclusions $\rho\leq\delta\leq \theta$ and
$\rho\not\leq\theta$ are contradictory. This proves Item (2).
\end{proof}

\begin{lm} \label{tau}
  Let $\m a$ be a finite algebra in a congruence modular
  variety. Assume also that $\m a$
  has centralizers split at $0$, and 
  that a particular relevant triple $(\delta,\theta,\nu)$
  is split by a particular triple $(\alpha,\beta,0)$.
  If $\tau\in\Con(\m a)$ is a completely meet irreducible congruence,
  then either
  \begin{itemize}
  \item $\tau\geq\alpha$, or
  \item $\tau\geq [\nu,\nu]$.
  \end{itemize}    
  In the latter case, $\tau\vee\rho\geq\nu$.
\end{lm}

\begin{proof}
Suppose that the upper cover of $\tau$ is $\mu$.
Assume that the first bulleted item does not hold,
  that is that $\tau\not\geq\alpha$. Then $\tau<\tau\vee\alpha$, so
  both $\mu$ and $\alpha$ are dominated by $\tau\vee\alpha$.
  Hence $\mu\vee\alpha\leq \tau\vee\alpha\leq \mu\vee\alpha$, 
yielding $\tau\vee\alpha=\mu\vee\alpha$. If we also had
  $\tau\wedge \alpha = \mu\wedge \alpha$, then $\{\alpha,\tau,\mu\}$
  would generate a pentagon in $\Con(\m a)$, contradicting
  modularity. Thus
  $\tau\wedge \alpha < \mu\wedge \alpha$, and by modularity we even
  have that the interval $I[\tau,\mu]$ is perspective with
  $I[\tau\wedge \alpha,\mu\wedge\alpha]$. This implies that
  $(\tau:\mu)=(\tau\wedge \alpha:\mu\wedge\alpha)$.
  But $\alpha\leq (\tau\wedge\alpha:\mu\wedge\alpha)$, since
  \[
    [\alpha,\mu\wedge\alpha]\leq [\alpha,\alpha]=0\leq \tau\wedge\alpha,
  \]
 and $\beta\leq (\tau\wedge\alpha:\mu\wedge\alpha)$, since
  \[
    [\beta,\mu\wedge\alpha]\leq [\beta,\alpha]\leq \beta\wedge\alpha
    =0\leq \tau\wedge\alpha.
  \]
Thus
$\nu=\alpha\vee\beta \leq (\tau\wedge\alpha:\mu\wedge\alpha)=(\tau:\mu)$.

By (C1)${}^*$ in $\m a/\tau$ we have the second $\leq$ in 
\[
[\nu,\nu]\leq[(\tau:\mu),(\tau:\mu)]\leq \tau.
\]
This shows that when 
$\tau\not\geq\alpha$ we have $\tau\geq [\nu,\nu]$, as claimed.
Moreover, when $\tau\geq [\nu,\nu]$ we calculate that
\[
\tau\vee\rho\geq [\nu,\nu]\vee\rho=\rho([\nu,\nu])\geq \nu,
\]
where the equality holds by Lemma~\ref{properties}~(4).
This justifies the last sentence of the lemma statement.
\end{proof}

Now we are prepared to prove one
direction of Theorem~\ref{split_cent_thm}.

\begin{thm}
  Let $\m a$ be a finite algebra in a congruence modular
  variety. If $\m a$ has centralizers
  split at $0$, then $\m a$ is neutrabelian.
\end{thm}

\begin{proof}
  We must show that if $\m a$ has centralizers split at $0$,
  then every SI quotient of $\m a$ is a neutrabelian SI.
  Let $\tau$ be a completely meet irreducible congruence
  of $\m a$ with upper cover $\mu$.
  We shall argue that $\m a/\tau$ is a neutrabelian SI.

As a first case, assume that
$\m a/\tau$ has nonabelian monolith.
That is, $\tau\stackrel{{\bf 3}, {\bf 4}}{\prec}\mu$.
It follows from
Lemma~\ref{properties}~(2) (transfer principles hold)
that $\m a/\tau$ is neutral, hence neutrabelian,
and we are done in this case.

Next we tackle the case where 
$\tau\stackrel{{\bf 2}}{\prec}\mu$.
It is a consequence of 
Lemma~\ref{properties}~(1)
and the fact that the class of algebras
in congruence modular varieties that 
satisfy (C1) is closed under quotients
that

\begin{clm} \label{C1}
$\m a/\tau$ satisfies {\rm(C1)}. \hfill$\diamond$
\end{clm}  

We also have that

\begin{clm} \label{comparable}
$\rho(\tau)$  is comparable with all congruences in
    the interval $I[\tau, 1]$.
\end{clm}

\begin{proof}[Proof of Claim]
This proof will involve the introduction of several congruences.
The order relationships between these congruences are
indicated in Figure~\ref{fig3}.

Assume there exists a congruence in $I[\tau,1]$
that is not comparable with $\rho(\tau)$.
Choose such 
a congruence $\psi\geq \tau$ that is minimal for the property
that $\psi$ is not comparable with $\rho(\tau)$.
Necessarily any lower cover $\psi_*$ of $\psi$
which lies in $I[\tau,1]$
will be strictly below $\rho(\tau)$,
hence $\psi_*:=\psi\wedge\rho(\tau)$
must be the unique lower cover of $\psi$ that
lies in $I[\tau,1]$.
Choose $\gamma$ so that $\psi_*\prec\gamma\leq \rho(\tau)$.
Since $\tau\leq\psi_*\prec\gamma\leq \rho(\tau)$
is a solvable interval, $\psi_*\stackrel{{\bf 2}}{\prec}\gamma$.

Extend $\psi$ to a congruence $\delta$ maximal for
$\delta\not\geq\gamma$. This congruence will be
completely meet irreducible. Its upper cover will be
$\theta:=\delta\vee \gamma$. Since the interval $I[\delta,\theta]$
is perspective with $I[\psi_*,\gamma]$ we must have
$\delta\stackrel{{\bf 2}}{\prec}\theta$.
Let $\nu = (\delta:\theta)$.
The triple
$(\delta,\theta,\nu)$ is relevant,
and $\m a$ has centralizers split at $0$, so 
assume that $(\delta,\theta,\nu)$ is split by $(\alpha,\beta,0)$.
That is, $\beta\leq\delta$, $\alpha$ is abelian,
$\alpha\wedge\beta = 0$ and $\alpha\vee\beta = \nu$.

\begin{figure}[H]
\setlength{\unitlength}{1truemm}
\begin{picture}(90,70)
\put(13,0){%
\begin{tikzpicture}[scale=.65]
\draw[line width=1.2pt] (0,-5) -- (-4,-1) -- (-1.4,1.6);
\draw[line width=1.2pt] (0,-5) -- (4,-1) -- (0,3);
\draw[line width=1.2pt] plot [smooth,tension=1]
               coordinates{(-1.4,1.6) (-.9,2.5) (0,3)};
\draw[line width=1.2pt] plot [smooth,tension=1]
               coordinates{(-1.4,1.6) (-.5,2.1) (0,3)};
\draw[line width=1.2pt] plot [smooth,tension=1]
               coordinates{(-1.4,1.6) (-1.18,3.08) (0,4)};
\draw[line width=1.2pt] plot [smooth,tension=1]
               coordinates{(-1.4,1.6) (.2,2.5) (0,4)};
\draw [fill] (-4,-1) circle [radius=0.1];    
\draw [fill] (-1.8,1.2) circle [radius=0.1]; 
\draw [fill] (-1.4,1.6) circle [radius=0.1];   
\draw [fill] (0,3) circle [radius=0.1];   
\draw [fill] (4,-1) circle [radius=0.1];
\draw [fill] (1,2) circle [radius=0.1];
\draw [fill] (3.3,-.3) circle [radius=0.1];
\node at (-4.4,-1){\small $\beta$};
\node at (-2.2,1.25){\small $\delta$};
\node at (-1.8,1.65){\small $\theta$};
\node at (0,3.3){\small $\nu$};
\node at (4.4,-1){\small $\alpha$};
\node at (1.8,2.1){\small $\rho(\tau)$};
\node at (3.7,-.2){\small $\rho$};
%
\draw[line width=1.2pt] plot [smooth,tension=1]
               coordinates{(1,2) (.2,-.2) (1,-2.4)};
\draw[line width=1.2pt] plot [smooth,tension=1]
               coordinates{(1,2) (1.8,-.2) (1,-2.4)};
\draw[line width=1.2pt] (1,-2.4) -- (1,-3.1);
\draw[line width=1.2pt] (-1.4,1.6) -- (.2,0);
\draw[line width=1.2pt] (-1.8,1.2) -- (.3,-.9);
\draw [fill] (1,-2.4) circle [radius=0.1];
\draw [fill] (1,-3.1) circle [radius=0.1];
\draw [fill] (.2,0) circle [radius=0.1];
\draw [fill] (.3,-.9) circle [radius=0.1];
\draw [fill] (-.3,-.3) circle [radius=0.1];
\node at (1.4,-2.4){\small $\mu$};
\node at (1.4,-3.1){\small $\tau$};
\node at (.6,0){\small $\gamma$};
\node at (.8,-.9){\small $\psi_*$};
\node at (-.7,-.4){\small $\psi$};

\draw[line width=1.2pt] plot [smooth,tension=1.5]
               coordinates{(0,-5) (-5,-.5) (0,4)};
\draw[line width=1.2pt] plot [smooth,tension=1.5]
               coordinates{(0,-5) (5,-.5) (0,4)};
\draw [fill] (0,4) circle [radius=0.1];
\draw [fill] (0,-5) circle [radius=0.1];
\node at (0,4.5){$1$};
\node at (0,-5.5){$0$};
\end{tikzpicture}
}
\end{picture}
\caption{}\label{fig3}
\end{figure}

According to Lemma~\ref{tau}, either
$\tau\geq\alpha$ or $\tau\geq [\nu,\nu]$.
In the former case we have
$\alpha\leq \tau\leq\psi\leq\delta$.
Since we also have $\beta\leq\delta$,
we derive that
$\theta\leq \nu=\alpha\vee\beta\leq\delta$, which is
false.
In the latter case we have
$[\nu,\nu]\leq\tau\leq \nu$
and $\nu\leq\tau\vee\rho=\rho(\tau)$. But then
$\psi\leq \delta\leq \nu\leq \rho(\tau)$,
contrary to our assumption that $\psi$ is not comparable
to $\rho(\tau)$.
This proves the claim.
\qedsymbdiamond
\end{proof}

It is a consequence of Lemma~\ref{properties}~(4)
that 

\begin{clm} \label{abel}
  $I[\tau,\rho(\tau)]$ is abelian and
  $I[\rho(\tau),1]$ is neutral. \hfill$\diamond$
\end{clm}
  
According to Lemma~\ref{neut_char}, Claims~\ref{C1},
\ref{comparable}, and
\ref{abel} establish that 
$\m a/\tau$ is neutrabelian. This completes the proof.
\end{proof}

Finally, we prove the other
direction of Theorem~\ref{split_cent_thm}.

\begin{thm}
  Let $\m a$ be a finite algebra in a congruence modular
  variety. If $\m a$ is neutrabelian, then $\m a$ has centralizers
  split at $0$.
\end{thm}

\begin{proof}
  Assume that $\m a$ is neutrabelian and that
$(\delta,\theta,\nu)$ is a relevant triple
  of $\m a$. Since $\delta$ is a completely meet irreducible
  congruence of $\m a$, $\m a/\delta$ is a neutrabelian SI.
  Moreover, since $\theta$ is an abelian cover of $\theta$,
$\m a/\delta$ has abelian monolith $\theta/\delta$.  

  By Lemma~\ref{properties}~(1), $\m a$ satisfies (C1)
  so the centralizer of $\theta/\delta$ is abelian.
  This congruence has the form $\nu/\delta$
  for $\nu:=(\delta:\theta)$. The fact that
  $\nu/\delta$ is abelian in $\Con(\m a/\delta)$
  translates exactly into the property that
  $[\nu,\nu]\leq\delta$ in $\Con(\m a)$.
Let $\beta = [\nu,\nu]$. The interval $I[\beta,\nu]$
is abelian (hence solvable), while the interval
$I[\nu,1]$ is neutral (since $\m a/\delta$
is neutrabelian and $\nu=(\delta:\theta)$),
so $\rho(\beta)=\nu$.

By Lemma~\ref{properties}~(3) and (4),
the solvable radical $\rho$ of $\m a$ is the
largest abelian congruence of $\m a$,
and $\rho\vee\beta=\rho(\beta)=\nu$. Our goal
in the rest of this proof is to find
$\alpha\leq \rho$ such that
\begin{enumerate}
\item $\alpha\vee\beta=\nu$, and
\item $\alpha\wedge\beta = 0$.
\end{enumerate}
Such an $\alpha$ will be abelian (since $\alpha\leq\rho$)
and then $(\alpha,\beta,0)$ will satisfy all conditions
necessary to be a triple that splits $(\delta,\theta,\nu)$
at zero.

We already have $\rho\vee\beta=\nu$. Shrink $\rho$
to an $\alpha\leq\rho$ still satisfying (1) $\alpha\vee\beta=\nu$
for which $\alpha\wedge\beta$ is minimal in $\Con(\m a)$.
We will argue that (2) must be satisfied by this congruence $\alpha$.
To obtain a contradiction, assume instead that
$\alpha\wedge\beta > 0$, and choose some $\sigma\prec \alpha\wedge\beta$.
Since
$\sigma <\alpha\wedge\beta\leq\alpha\leq \rho$ and $\rho$ is abelian,
the covering $\sigma\prec \alpha\wedge\beta$ is abelian.
Choose $\psi\geq \sigma$ maximal for $\psi\not\geq \alpha\wedge\beta$,
and let $\psi^* = \psi \vee (\alpha\wedge\beta)$.
The quotient $\m a/\psi$ is subdirectly irreducible with
monolith $\psi^*/\psi$. The interval
$I[\psi,\psi^*]$ is perspective with
the abelian interval $I[\sigma,\alpha\wedge\beta]$,
so the subdirectly irreducible quotient
$\m a/\psi$ has abelian monolith.

We shall argue that, for $\alpha_{*}:=\alpha\wedge\psi\leq\alpha$,
we have (i)~$\beta\wedge \alpha_{*} <\beta\wedge \alpha$
and (ii)~$\beta\vee\alpha_{*}=\nu$, contradicting the minimality
of $\beta\wedge\alpha$. We rewrite the necessary conditions
as (i)~$\beta\wedge (\alpha\wedge\psi) <\beta\wedge \alpha$
and (ii)~$\beta\vee(\alpha\wedge \psi)=\nu$.

\begin{clm} \label{claim1}
{\rm(i)} $(\alpha\wedge\psi)\wedge\beta = \sigma \prec\alpha\wedge\beta$.
\end{clm}

\begin{proof}[Proof of Claim]
We chose $\sigma$ so that $\sigma\prec \alpha\wedge\beta$
and we chose $\psi$ so that $\sigma\leq \psi$ and 
$\alpha\wedge\beta\not\leq \psi$. This is enough to conclude
that $(\alpha\wedge\beta)\wedge\psi = \sigma$.
\qedsymbdiamond
\end{proof}

\begin{clm} \label{claim2}
{\rm(ii)} $\beta\vee(\alpha\wedge \psi)=\nu$.
\end{clm}

\begin{proof}[Proof of Claim]
We shall derive this claim through an examination
of the sublattice of
$\Con(\m a)$ that is generated by $X=\{\alpha,\beta,\psi\}$.
This is a $3$-generated modular lattice, so it has at most
$28$ elements, and it can easily be drawn.
Some elements of this sublattice have already been 
named above, for example $\alpha\vee\beta = \nu$
and $\alpha\wedge\beta\wedge\psi = \sigma$ (the
least element of the sublattice).

\begin{figure}
\setlength{\unitlength}{1truemm}
\begin{picture}(80,70)
\put(10,0)
{
\begin{tikzpicture}[scale=.8]
\draw [fill] (0,4) circle [radius=0.1];
\draw [fill] (1,3) circle [radius=0.1];
\draw [fill] (1,5) circle [radius=0.1];
\draw [fill] (2,1) circle [radius=0.1];
\draw [fill] (2,2) circle [radius=0.1];
\draw [fill] (2,4) circle [radius=0.1];
\draw [fill] (2,6) circle [radius=0.1];
\draw [fill] (2,7) circle [radius=0.1];
\draw [fill] (3,0) circle [radius=0.1];
\draw [fill] (3,1) circle [radius=0.1];
\draw [fill] (3,2) circle [radius=0.1];
\draw [fill] (3,3) circle [radius=0.1];
\draw [fill] (3,4) circle [radius=0.1];
\draw [fill] (3,5) circle [radius=0.1];
\draw [fill] (3,6) circle [radius=0.1];
\draw [fill] (3,7) circle [radius=0.1];
\draw [fill] (3,8) circle [radius=0.1];
\draw [fill] (4,1) circle [radius=0.1];
\draw [fill] (4,2) circle [radius=0.1];
\draw [fill] (4,4) circle [radius=0.1];
\draw [fill] (4,6) circle [radius=0.1];
\draw [fill] (4,7) circle [radius=0.1];
\draw [fill] (5,3) circle [radius=0.1];
\draw [fill] (5,5) circle [radius=0.1];
\draw [fill] (6,4) circle [radius=0.1];

\draw [fill] (1,4) circle [radius=0.1];
\draw [fill] (2,3) circle [radius=0.1];
\draw [fill] (2,5) circle [radius=0.1];

\draw[line width=1.2pt] (3,1) -- (0,4) -- (3,7) -- (6,4) -- (3,1);
\draw[line width=1.2pt] (2,2) -- (1,3) -- (4,6) -- (5,5) -- (2,2);
\draw[line width=1.2pt] (4,2) -- (5,3) -- (2,6) -- (1,5) -- (4,2);
\draw[line width=1.2pt] (1,4) -- (4,7) -- (3,8) -- (2,7) -- (3,6) -- (3,2) -- (2,1) -- (3,0) -- (4,1) -- (1,4);
\draw[line width=1.2pt] (2,5) -- (3,4) -- (2,3);
\draw[line width=1.2pt] (2,1) -- (2,2);
\draw[line width=1.2pt] (3,0) -- (3,1);
\draw[line width=1.2pt] (4,1) -- (4,2);
\draw[line width=1.2pt] (2,6) -- (2,7);
\draw[line width=1.2pt] (3,7) -- (3,8);
\draw[line width=1.2pt] (4,6) -- (4,7);

\node at (-.4,4){$\beta$};
\node at (.7,4){$\alpha$};
\node at (1.5,7){$\nu$};
\node at (6.5,4){$\psi$};
\node at (5.3,5.3){$\psi^*$};
\node at (1.2,1){$\alpha\wedge\beta$};
\node at (3,-.5){$\sigma$};
\node at (4.8,1){$\alpha\wedge\psi$};
\end{tikzpicture}
}
\end{picture}
\caption{}\label{fig4}
\end{figure}

The sublattice of
$\Con(\m a)$ that is generated by $X=\{\alpha,\beta,\psi\}$
is a quotient of the modular
lattice freely generated by the set $X$, which is drawn in Figure~\ref{fig4}.
But the sublattice of our lattice
$\Con(\m a)$ that is generated
by $X$ is not \emph{freely} generated by $X$, it is a proper quotient
of the lattice in Figure~\ref{fig4}. In particular,
the sublattice of $\Con(\m a)$ generated by $X$
must satisfy the following relations.

\begin{subclm} \label{subclaim}
  \mbox{}

\begin{itemize}
    \item $\beta\vee\psi=\alpha\vee\beta\vee\psi$, and
    \item $\alpha\vee\psi=(\alpha\wedge\beta)\vee\psi$.
\end{itemize}
\end{subclm}

To prove the first bulleted item note that,
since $\beta=[\nu,\nu]$, the interval
from $[\beta,\beta]$ to $\nu$ is $2$-step
solvable. Hence $\nu \leq \rho([\beta,\beta])= [\beta,\beta]\vee\rho$,
implying that the interval
from $[\beta,\beta]$ to $\nu$ is abelian. Therefore
$\beta=[\nu,\nu]\leq [\beta,\beta]\leq \beta$, or just
$\beta=[\beta,\beta]$.
From this we derive from {\rm (C1)} that if
$x\leq \beta$, then
\[
x = x\wedge\beta = x\wedge [\beta,\beta] = [x\wedge \beta,\beta] = [x,\beta].
\]
In particular, when $x=\alpha\wedge\beta$ we get that
$[\alpha\wedge\beta,\beta]=\alpha\wedge\beta\not\leq \sigma$,
which we can express as $\beta\not\leq (\sigma:\alpha\wedge\beta)$.
But the interval $I[\sigma:\alpha\wedge\beta]$ is perspective
with $I[\psi,\psi^*]$, so $(\sigma:\alpha\wedge\beta)=(\psi:\psi^*)$,
and we derive that $\beta\not\leq (\psi:\psi^*)$.

The SI quotient $\m a/\psi$ is neutrabelian, and
$(\beta\vee\psi)/\psi\not\leq (0:\psi^*/\psi)$,
so $\beta\vee\psi > \rho(\psi)$.
Since $\nu\geq \beta$ and $\nu\solv \beta$
we have $(\nu\vee\psi)/\psi\geq (\beta\vee\psi)/\psi$ and
$(\nu\vee\psi)/\psi \solv (\beta\vee\psi)/\psi$.
But solvably related congruences above the radical
of a neutrabelian SI are equal,
so $(\nu\vee\psi)/\psi = (\beta\vee\psi)/\psi$. Hence
$\nu\vee\psi=\beta\vee\psi>\rho(\psi)=\rho\vee\psi$.
Now $\nu=\beta\vee\alpha$, so we get
$\alpha\vee\beta\vee\psi = \nu\vee\psi=\beta\vee\psi$,
completing the proof of the first bulleted item
of the subclaim.

For the second bulleted item of the subclaim, we carry along
our earlier conclusions that 
$\alpha\vee\beta\vee\psi=\beta\vee\psi > \rho(\psi)$.
This strict inequality means that $(\beta\vee\psi)/\psi$ is strictly
above the radical of the neutrabelian SI quotient
$\m a/\psi$. From the way the commutator behaves
on a neutrabelian algebra, 
$(\beta\vee\psi)/\psi$ must fail to centralize
any proper interval in $\Con(\m a/\psi)$ that lies below 
$(\beta\vee\psi)/\psi$.
In particular, if 
$(\alpha\wedge\beta)\vee\psi<\alpha\vee\psi$, then
we must have
\[
  [(\beta\vee\psi)/\psi,(\alpha\vee\psi)/\psi]\not\leq ((\alpha\wedge\beta)\vee\psi)/\psi.
  \]
But this is in contradiction with basic properties of the commutator
(additivity, submultiplicativity, homomorphism property),
as we see in the calculation
\[
[\beta\vee\psi,\alpha\vee\psi]\vee\psi = [\beta,\alpha]\vee\psi\leq (\alpha\wedge\beta)\vee\psi.
\]
This proves that second bulleted item of the subclaim.

\bigskip

From Subclaim~\ref{subclaim} we know that the sublattice of
$\Con(\m a)$ generated by $X=\{\alpha,\beta,\psi\}$
is a homomorphic image of the modular lattice presented by
\[
\langle \alpha, \beta, \psi\;|\;
\beta\vee\psi=\alpha\vee\beta\vee\psi,\;
\alpha\vee\psi=(\alpha\wedge\beta)\vee\psi\rangle.
\]
This lattice can be constructed as a quotient of the lattice
in Figure~\ref{fig4}, and it is drawn in Figure~\ref{fig5}.
Observe that $\beta\vee(\alpha\wedge\psi)=\nu$
in this lattice, so this relation
will hold in any quotient, such as the sublattice
$\Con(\m a)$ generated by $X=\{\alpha,\beta,\psi\}$.
This ends the proof of Claim~\ref{claim2}.
\qedsymbdiamond
\end{proof}

\begin{figure}
\setlength{\unitlength}{1truemm}
\begin{picture}(80,50)
\put(16,0)
{
\begin{tikzpicture}[scale=.8]
\draw [fill] (1,3) circle [radius=0.1];
\draw [fill] (2,1) circle [radius=0.1];
\draw [fill] (2,2) circle [radius=0.1];
\draw [fill] (2,4) circle [radius=0.1];
\draw [fill] (3,0) circle [radius=0.1];
\draw [fill] (3,1) circle [radius=0.1];
\draw [fill] (3,2) circle [radius=0.1];
\draw [fill] (3,3) circle [radius=0.1];
\draw [fill] (3,5) circle [radius=0.1];
\draw [fill] (4,1) circle [radius=0.1];
\draw [fill] (4,2) circle [radius=0.1];
\draw [fill] (4,4) circle [radius=0.1];
\draw [fill] (5,3) circle [radius=0.1];

\draw[line width=1.2pt] (3,1) -- (5,3) -- (3,5) -- (1,3) -- (3,1);
\draw[line width=1.2pt] (2,4) -- (4,2) -- (4,1) -- (3,0) -- (2,1) -- (2,2) -- (4,4);
\draw[line width=1.2pt] (3,0) -- (3,1);
\draw[line width=1.2pt] (3,2) -- (3,3);
\draw[line width=1.2pt] (2,1) -- (3,2) -- (4,1);

\node at (.5,3){$\beta$};
\node at (3.4,2){$\alpha$};
\node at (1.5,4){$\nu$};
\node at (5.5,3){$\psi$};
\node at (4.3,4.3){$\psi^*$};
\node at (1.2,1){$\alpha\wedge\beta$};
\node at (3,-.5){$\sigma$};
\node at (4.8,1){$\alpha\wedge\psi$};
\end{tikzpicture}
}
\end{picture}
\caption{}\label{fig5}
\end{figure}

Claims~\ref{claim1} and \ref{claim2}
contradict the minimality of $\alpha\wedge\beta$
in the situation where $\alpha\wedge\beta>0$,
so we conclude that $\alpha\wedge\beta=0$.
This shows that the relevant triple
$(\delta,\theta,\nu)$ is split at $0$
by $(\alpha,\beta,0)$.
\end{proof}

\bibliographystyle{plain}

\end{document}